\documentclass[14pt]{amsart}

\usepackage{amsmath,amssymb,amsfonts,enumerate,amsthm, amscd}
\usepackage[french, ngerman, italian, english]{babel}
\usepackage[colorlinks=true]{hyperref}
\usepackage{paralist}
\usepackage{verbatim}
\usepackage{color}
\usepackage{xcolor}
\usepackage{tikz}
\usetikzlibrary{arrows,decorations.pathmorphing,backgrounds,positioning,fit}
\RequirePackage{pgfrcs}

\theoremstyle{plain}
\newtheorem{thm}{\bf Theorem}[section]

\newtheorem{prop}[thm]{\bf Proposition}
\newtheorem{lem}[thm]{\bf Lemma}

\theoremstyle{definition}
\newtheorem{defn}[thm]{\bf Definition}
\theoremstyle{remark}
\newtheorem{rem}[thm]{\bf Remark}

\newtheorem{conj}{\bf Conjecture}
\theoremstyle{example}

\def \Shad{{\mathrm{Shad}}}
\def \supp{{\mathrm{supp}}}

\def \n{\mathbf n}

\def \1{\mathbf 1}

\begin{document}

\title[]{On the polymatroidal property of monomial ideals with a view towards orderings of minimal generators}

\author{Somayeh Bandari and Rahim Rahmati-Asghar}
\address{Somayeh Bandari, Department of Mathematics, Buein Zahra Technical University,  Buein Zahra, Qazvin,
Iran.} \email{somayeh.bandari@yahoo.com}

\address{Rahim Rahmati-Asghar, Department of Mathematics, Faculty of Basic Sciences,
University of Maragheh, P. O. Box 55136-553, Maragheh, Iran and
School of Mathematics, Institute for Research in Fundamental
Sciences (IPM), P.O. Box 19395-5746, Tehran, Iran.}
\email{rahmatiasghar.r@gmail.com}



 \maketitle

\begin{abstract}
We prove that a monomial ideal $I$ generated in a single degree, is
polymatroidal if and only if it has linear quotients with respect to
the lexicographical ordering of the minimal generators induced by
every ordering of variables. We also conjecture that the
polymatroidal ideals can be characterized with linear quotients
property   with respect to the reverse lexicographical ordering of
the minimal generators induced by every ordering of variables. We
prove our conjecture in many special cases.
\end{abstract}

\section{\bf Introduction}

A monomial ideal is called polymatroidal, if its monomial generators
correspond to the bases of a discrete polymatroid, see \cite{HeHi}.
Since the set of bases of a discrete polymatroid is characterized by
the so-called exchange property, it follows that a polymatroidal
ideal may as well be characterized as follows: let $I\subset
S=K[x_1,\ldots,x_n]$ be a monomial ideal generated in a single
degree and  $G(I)$ be the unique minimal set of monomial generators
of $I$. Then $I$ is said to be polymatroidal, if for any two
elements $u,v\in G(I)$ such that $\deg_{x_i}(u)> \deg_{x_i}(v)$
there exists an index $j$ with $\deg_{x_j}(u)< \deg_{x_j}(v)$ such
that $x_j(u/x_i)\in I$. A squarefree polymatroidal ideal is called
matroidal. They have many attractive structural properties which
have been considered by many mathematicians in recent years, see for
example \cite{BaHe,CoHe,HeHi06,HeRaVl}. Bandari and Herzog
\cite{BaHe} conjectured  that a monomial ideal is polymatroidal if
and only if all its monomial localizations have a linear resolution.
In this paper we try to characterize polymatroidal ideals with
linear quotients property.
 Bj\"{o}rner in
\cite[Theorem 7.3.4]{Bj} showed that a simplicial complex $\Delta$
is a matroid complex if and only if $\Delta$ is pure and every
ordering of the vertices induces a shelling. Hence a monomial ideal
 is matroidal if and only if it is  generated in a single degree and it also has linear
 quotients with
respect to every ordering
 of generators.
  Herzog and Takayama in \cite{{HeTa}}
showed that a polymatroidal ideal has linear quotients with respect
to the reverse lexicographical order of the minimal generators
induced by every ordering of variables. A natural question arises
whether we can generalize Bj\"{o}rner's result to polymatroidal
ideals. In Proposition \ref{proplex}, we show that a polymatroidal
ideal also has linear quotients with respect to the  lexicographical
order of the minimal generators. Then in Theorem \ref{thmlex}, we
prove that for a monomial ideal $I$  generated in a single degree,
$I$ is polymatroidal if and only if
 $I$ has linear quotients with respect to the  lexicographical ordering of the minimal
generators induced by every ordering of variables.

It is natural to ask whether we have similar result  with the
reverse lexicographical order assumption.
  Due to
computational evidence we are lead to conjecture that the monomial
ideals generated in a single degree which  have linear quotients
with respect to the reverse lexicographical order of the generators
induced by every ordering of variables are precisely the
polymatroidal ideals. We discuss several special cases which support
this conjecture. In fact we give a positive answer to the conjecture
in the following cases:
 1. $I$ is generated in degree 2 (Proposition \ref{quadratic}), 2. $I$ contains at
least $n-1$ pure powers (Proposition \ref{pure power}), 3. $I$ is
monomial ideal in at most 3 variables (Remark \ref{two variable} and
Proposition \ref{dimthree}), 4. $I=(u\in M_d\;|\; u\geq_{lex} v)$
for some $v\in M_d$ (Proposition \ref{lex}), 5. $I$ is a completely
lexsegment ideal (Proposition \ref{comlex}).

\section{\bf The polymatroidality and orderings induced by variables on minimal generators}
Throughout $S=K[x_1,\ldots,x_n]$ is the polynomial ring over a field
$K$, $>_{lex}$ is  the  lexicographic order  and $>$ is the reverse
lexicographic order on monomials of $S$. In the following, we recall
some preliminary concepts.

Let $u=x_1^{a_1}\cdots x_n^{a_n}$ be a monomial of $S$. We define
$\deg(u)=\sum_{i=1}^na_i$, $\deg_{x_i}(u)=a_i$ and
$\supp(u)=\{x_i\;|1\leq i\leq n\;,\; a_i>0\}$.

Let $u=x_1^{a_1}\cdots x_n^{a_n}$ and $v=x_1^{b_1}\cdots x_n^{b_n}$.
 We define the total order $>_{lex}$ on monomials of
$S$ by setting $u>_{lex}v$ if either (i) $\deg(u)>\deg(v)$ or (ii)
$\deg(u)=\deg(v)$ and there exists an integer $i$ with $a_i>b_i$
 and $a_k = b_k$ for $k=1,\ldots,i-1$. It
follows that $>_{lex}$ is a monomial order on $S$, which is called
the \emph{lexicographic order} on $S$ induced by the ordering
$x_1>_{lex}x_2
>_{lex}\cdots>_{lex} x_n$.

 We define the total order $>$ on monomials of
$S$ by setting $u>v$ if either (i) $\deg(u)>\deg(v)$ or (ii)
$\deg(u)=\deg(v)$ and there exists an integer $i$ with $b_i>a_i$
 and $a_k = b_k$ for $k=i+1,\ldots,n$. It
follows that $>$ is a monomial order on $S$, which is called the
\emph{reverse lexicographic order} on $S$ induced by the ordering
$x_1>x_2
>\cdots> x_n$.

For a monomial ideal $I$ and a monomial $v$ of $S$,
$\{u/gcd(u,v)\;|\; u\in G(I)\}$ is a set of generators of $I:v$.
\begin{defn}
The monomial ideal $I\subset S$ has \emph{linear quotients} whenever
there is an ordering $u_1,\ldots,u_r$ on the minimal generators of
$I$ such that for $j=2,\ldots,r$, the minimal generators of the
colon ideal $(u_1,\ldots,u_{j-1}):u_j$ are variables.
\end{defn}

 By \cite[Lemma 1.3]{HeTa} and
\cite[Lemma 4.1]{CoHe}, a polymatroidal ideal $I$ has linear
quotients with respect to the reverse lexicographical order of the
minimal generators and so it has a linear resolution.
 In the following, we show that we have similar result with  the  lexicographical order assumption.

\begin{prop}\label{proplex}
Let $I$ be a  polymatroidal ideal. Then $I$  has linear quotients
with respect to the  lexicographical order of the minimal
generators.
\end{prop}

\begin{proof}
 Let $u\in G(I)$ and  $J=(v\in G(I)\;|\;v>_{lex}u)$. In order to prove that $J:u$ is generated
by monomials of degree 1, we have to show that for each $v\in G(I)$
such that $v>_{lex}u$, there exists $x_i\in  J:u$ such that
$x_i|v:u$. Let $u=x_1^{a_1}\cdots x_n^{a_n}$  and $v=x_1^{b_1}\cdots
x_n^{b_n}$. Since $v>_{lex}u$, there exists an integer $i$ with
$b_i>a_i$
 and $a_k = b_k$ for $k=1,\ldots,i-1$. Now since $I$ is a
polymatroidal ideal, it follows by \cite[Theprem 3.1]{HeHi06} that
there exists an integer $j$ with $b_j<a_j$ such that
$u'=(u/x_j)x_i\in G(I)$. Since $j>i$, we have that $u'\in J$. Hence
$x_iu=x_ju'\in J$, so $x_i\in J:u$. On the other hand since
$\deg_{x_i}(v:u)=b_i-\min\{b_i,a_i\}=b_i-a_i>0$, we have $x_i|v:u$.
\end{proof}

For proof of the next theorem, we need the following result.
\begin{lem}\label{pure}
Let $I\subset S$ be a monomial ideal generated in a single degree
and $u,v\in G(I)$ such that $\supp(u:v)=\{x_1\}$. Suppose $I$ has
linear quotients with respect to the  lexicographical order
$>_{lex}$, induced by $x_1>_{lex}\cdots>_{lex}x_n$. Then we have the
following exchange property:
 $$(u/x_1)x_i\in I \;\text{for some}\; i\; \text{with}\; \deg_{x_i}(v)>\deg_{x_i}(u).$$
\end{lem}

\begin{proof}
Let $u:v=x_1^m$  for some $m\geq 1$. By induction on $m$, we show
that we have the exchange property. The assertion is trivial, if
$m=1$. Now, let $m\geq 2$. Since $u>_{lex}v$, it follows by linear
quotient property that there exists a monomial $w\in G(I)$, such
that $w>_{lex} v$ and $w:v=x_1$. Indeed $w=x_1z$ and $v=x_tz$ for
some monomial $z$ and $t\neq 1$. Now since $u:w=x_1^{m-1}$, our
induction hypothesis implies that $(u/x_1)x_i\in I$, for some $i$
with $\deg_{x_i}(w)>\deg_{x_i}(u)$. Note that $i\neq 1$, that
implies
$\deg_{x_i}(u)<\deg_{x_i}(w)=\deg_{x_i}(z)\leq\deg_{x_i}(v)$.

\end{proof}
\begin{thm}\label{thmlex}
Let $I\subset S$ be a monomial ideal generated in a single degree.
Then the following conditions are equivalent:
\begin{enumerate}[\upshape (a)]
  \item $I$ is polymatroidal.
  \item $I$ has linear quotients with respect to the  lexicographical ordering of the minimal
  generators induced by every ordering of variables.
\end{enumerate}
\end{thm}
\begin{proof}
The implication (a) $\implies$ (b) holds by Proposition
\ref{proplex}.

(b)$\implies$ (a): Let $u\in G(I)$ with $\deg_{x_1}(u)>0$. We want
to show that for any monomial $v\in G(I)$ with
$\deg_{x_1}(u)>\deg_{x_1}(v)$ we have the following exchange
property:
 $$(u/x_1)x_i\in I \;\text{for some}\; i\; \text{with}\; \deg_{x_i}(v)>\deg_{x_i}(u).$$
Let $A$ be the set of those monomials $v\in G(I)$ such that
$\deg_{x_1}(v)<\deg_{x_1}(u)$ and $(u/x_1)x_i\notin I$, for all $i$
with $\deg_{x_i}(v)> \deg_{x_i}(u)$. We prove by contradiction that
$A$ is an empty set. Assume the opposite that $A$ is not empty. Let
$v_1\in A$ such that $\deg_{x_1}(u:v_1)\geq \deg_{x_1}(u:v)$, for
all $v\in A$. Now let
$$B=\{v\in A\;|\;\deg_{x_1}(u:v)=\deg_{x_1}(u:v_1)\}.$$
Consider the lexicographical order $>_{lex}$, induced by
$x_1>_{lex}\cdots>_{lex}x_n$. Let $v\in B$ such that
$$u:v=\min_{>_{lex}}\{u:v\;|\;\;v\in B\}=\min_{>_{lex}}\{u:v\;|\;\;v\in
A\;,\; \deg_{x_1}(u:v)=\deg_{x_1}(u:v_1)\}.$$  By Lemma \ref{pure},
we know that $|\supp(u:v)|\geq 2$. Let $t+1=\min\{i\geq
2\;;x_i|u:v\}$. Consider the lexicographical order $\succ_{lex}$,
induced by
$$x_{j_r}\succ_{lex}\cdots x_{j_1}\succ_{lex}x_{t+1}\succ_{lex}\cdots\succ_{lex}x_n\succ_{lex}x_{i_s}\succ_{lex}\cdots x_{i_1}\succ_{lex}x_1,$$
where
$$\{x_{i_1},\ldots,x_{i_s}\}=\{x_l\;|\;2\leq l\leq t \;,\;
\deg_{x_l}(v)>\deg_{x_l}(u)\},$$
$$\{x_{j_1},\ldots,x_{j_r}\}=\{x_l\;|\;2\leq l\leq t \;,\;
\deg_{x_l}(v)=\deg_{x_l}(u)\}.$$ Since $x_{t+1}|u:v$ and
$\deg_{x_l}(u)=\deg_{x_l}(v)$ for all $l\in\{j_1,\ldots,j_r\}$, it
follows that $u\succ_{lex}v$. By linear quotient property, there
exists a monomial $w\in G(I)$, such that $w\succ_{lex}v$,  $w:v=x_i$
and $x_i|u:v$ for some $i$. In particular $w=x_iz$ and  $v=x_jz$ for
some monomial $z$ and $x_i\succ_{lex}x_j$. Since $x_1$ is the
smallest variable with respect to $\succ_{lex}$, we have $i\neq 1$.
Since $u:v=u:(x_jz)=(u:z):x_j$, it follows that either $u:v=u:z$ if
$x_j\nmid u:z$ or $u:v=\frac{u:z}{x_j}$ if $x_j|u:z$. Now since
$u:w=(u:z):x_i$, it follows that either (i) $u:w=\frac{u:v}{x_i}$ or
(ii) $u:w=\frac{u:v}{x_i}x_j$.

Case (i): Let $u:w=\frac{u:v}{x_i}$. Then $u:v>_{lex}u:w$ and since
$i\neq 1$, it follows that $\deg_{x_1}(u:w)=\deg_{x_1}(u:v)$.
Therefore $w\notin A$.

Case (ii): Let $u:w=\frac{u:v}{x_i}x_j$. If $j=1$, then
$\deg_{x_1}(u:w)=\deg_{x_1}(u:v)+1>\deg_{x_1}(u:v)$. Hence $w\notin
A$. Now let  $j\geq 2$. Since $x_i|u:v$ and $i\neq 1$, it follow
that  $i\in\{t+1,\ldots,n\}$. On the other hand, since $x_j|u:w$, it
follows that $\deg_{x_j}(u)>\deg_{x_j}(w)=\deg_{x_j}(z)$, so
$\deg_{x_j}(u)\geq \deg_{x_j}(z)+1=\deg_{x_j}(v)$. Hence, since
$x_i\succ_{lex}x_j$ and $i\in\{t+1,\ldots,n\}$, we have
$j\in\{t+1,\ldots,n\}$ and  $x_i>_{lex}x_j$. Therefore $u:v>u:w$.
Now since $\deg_{x_1}(u:w)=\deg_{x_1}(u:v)$, it follows that
$w\notin A$.

In both cases (i) and (ii), we show $w\notin A$. On the other hand
$\deg_{x_1}(u)>\deg_{x_1}(v)\geq\deg_{x_1}(z)=\deg_{x_1}(w)$.
Therefore $(u/x_1)x_s\in I$, for some $s$ with
$\deg_{x_s}(w)>\deg_{x_s}(u)$. Since
$\deg_{x_i}(w)=\deg_{x_i}(z)+1=\deg_{x_i}(v)+1\leq \deg_{x_i}(u)$,
it follows that $i\neq s$. So
$\deg_{x_s}(v)\geq\deg_{x_s}(z)=\deg_{x_s}(w)>\deg_{x_s}(u)$. It
contradicts our assumption $v\in A$.

Replacing $x_1$ with $x_i$, the same argument proves the exchange
property for $I$.
\end{proof}

In the following, we discuss several special cases where the above
result is true, replacing the lexicographical order with the reverse
lexicographical order.

 The following result holds immediately by \cite[Corollary
2.5]{BaHe}.

\begin{rem}\label{two variable}
 Let $I\subset K[x_1,x_2]$ be a
monomial ideal generated in a single degree. Then the following
conditions are equivalent:
\begin{enumerate} [\upshape (a)]
  \item $I$ is polymatroidal.
  \item $I$ has linear quotients with respect to the reverse lexicographical ordering of the minimal
   generators induced by every ordering of variables.
   \item $I$ has a linear resolution.
\end{enumerate}
\end{rem}

\begin{prop}\label{quadratic}
Let $I\subset S$ be a monomial ideal generated in degree $2$. Then
the following conditions are equivalent:
\begin{enumerate} [\upshape (a)]
  \item $I$ is polymatroidal.
  \item $I$ has linear quotients with respect to the reverse lexicographical ordering of the minimal
   generators induced by every ordering of variables.
\end{enumerate}
\end{prop}

\begin{proof}
The implication (a) $\implies$ (b) is known.

(b) $\implies$ (a):   Let $u,v\in G(I)$ with
$\deg_{x_i}(u)>\deg_{x_i}(v)$. If $x_i|v$, there is nothing to
prove. Otherwise, we consider the following cases:

Case (i): Let $u=x_i^2$ and $v=x_jx_l$ (it can be $j=l$), so
$u:v=x_i^2$. Assume that $x_i>x_d$ for each $d\neq i$, so  $u>v$.
Hence by the assumption there exists $w\in G(I)$, such that $w>v$
and $w:v=x_i$. Hence either $w=x_ix_j=(u/x_i)x_j\in G(I)$ or
$w=x_ix_l=(u/x_i)x_l\in G(I)$.

Case (ii): Let $u=x_ix_t$ and $v=x_jx_l$ (it can be $j=l$). If
either $t=j$ or $t=l$, there is nothing to prove. Otherwise, assume
that $x_t>x_l>x_i$ and  $x_t>x_j>x_i$.  So $v>u$ and $v:u=x_jx_l$.
Hence by the assumption there exists $w\in G(I)$, such that $w>u$
and  we have either $w:u=x_j$ or $w:u=x_l$. If $w:u=x_j$, then
$w=x_tx_j=(u/x_i)x_j\in G(I)$. If $w:u=x_l$, then
$w=x_tx_l=(u/x_i)x_l\in G(I)$.
\end{proof}

In the following we recall the monomial localization concept, which
we will use it in the next proposition.

\begin{defn} Let $P$ be a monomial prime ideal of $S=K[x_1,\ldots,x_n]$.
Then $P=P_C$ for some subset $C\subset [n]$, where $P_C=(\{x_i\:\;
i\not\in C\})$
 and  $IS_P=JS_P$ where $J$ is the monomial ideal  obtained from $I$ by the substitution $x_i\mapsto 1$
 for all $i\in C$. We call $J$ the \emph{monomial localization} of $I$ with respect to $P$ and denote it by $I(P)$.
 \end{defn}

\begin{prop}\label{pure power}
Let $I\subset K[x_1,\ldots,x_n]$ be a monomial ideal generated in
degree $d$ and suppose that $I$ contains at least $n-1$ pure powers
of the variables, say $x_1^d,\ldots,x_{n-1}^d$. Then the following
conditions are equivalent:
\begin{enumerate}[\upshape (a)]
  \item $I$ is polymatroidal.
  \item $I$ has linear quotients with  respect to the reverse lexicographical ordering
  of the minimal generators induced by every ordering of variables.
\end{enumerate}
\end{prop}

\begin{proof}
The implication (a) $\implies$ (b) is known.

(b)$\implies$ (a): Let $k=max\{\deg_{x_n}(u) \;|\; u\in G(I)\}$ and
$w=x_1^{a _1}\cdots x_{n-1}^{ a_{n-1}}x_n^k\in G(I)$. We want to
show that $x_i^{d-k}x_n^k\in G(I)$ for all $1\leq i\leq n-1$. We fix
$i$. If $w=x_i^{d-k}x_n^k$, then there is nothing to prove.
Otherwise, for monomial $v$, we define $\deg_A(v)=\sum_{i\in
A}\deg_{x_i}(v)$, where $A:=\supp(w)\setminus\{x_i,x_n\}$. We
consider the ordering $x_n>x_i>x_j$ for all $j\neq i,n$. Since
$x_i^d>w$ and $x_i^d:w=x_i^{d-a_i}$, it follows by the assumption
that there exists $w_1\in G(I)$ such that $w_1>w$ and $w_1:w=x_i$.
Hence $w_1=(w/x_{j_1})x_i$, where $j_1\in A$,
$\deg_{x_i}(w_1)=a_i+1$ and $\deg_A(w_1)=\deg_A(w)-1$. Now since
$x_i^d:w_1=x_i^{d-(a_i+1)}$, it follows that there exists $w_2\in
G(I)$ such that $w_2>w_1$ and $w_2:w_1=x_i$. Hence
$w_2=(w_1/x_{j_2})x_i$, where $j_2\in A$, $\deg_{x_i}(w_2)=a_i+2$
and $\deg_A(w_2)=\deg_A(w)-2$. Continuing in the same way, there
exists $w_{h-1}\in G(I)$ where $h:=d-k-a_i$,
$\deg_{x_i}(w_{h-1})=a_i+h-1$ and  $x_i^d:w_{h-1}=x_i^{k+1}$. So
there exists $w_h\in G(I)$ such that $w_h>w_{h-1}$ and
$w_h:w_{h-1}=x_i$. Hence $w_h=(w_{h-1}/x_{j_h})x_i$ where $j_h\in
A$, $\deg_{x_i}(w_h)=a_i+h=d-k$ and $\deg_A(w_h)=\deg_A(w)-h=0$.
Therefore $w_h=x_i^{d-k}x_n^k$.

Now, by induction on $h$ we show $x_{i_1}^{\alpha_{i_1}}\cdots
x_{i_{h-1}}^{\alpha_{i_{h-1}}}x_{i_h}^{\alpha_{i_h}}x_n^k\in G(I),$
where $\alpha_{i_1}+\cdots+\alpha_{i_{h-1}}+\alpha_{i_h}=d-k$. For
$h=1$, we have already proved it. Now assume that $h>1$. We set
$u:=x_{i_1}^{\alpha_{i_1}}\cdots
x_{i_{h-1}}^{\alpha_{i_{h-1}}+\alpha_{i_h}}x_n^k$ and
$v:=x_{i_1}^{\alpha_{i_1}}\cdots
x_{i_{h-2}}^{\alpha_{i_{h-2}}}x_{i_h}^{\alpha_{i_{h-1}}+\alpha_{i_h}}x_n^k$.
By induction hypothesis, $u,v\in G(I)$. Assume that
$x_n>x_{i_1}>x_{i_2}>\cdots >x_{i_{h-1}}>x_{i_h}$. Since $u>v$ and
$u:v=x_{i_{h-1}}^{\alpha_{i_{h-1}}+\alpha_{i_h}}$, it follows that
there exists $w_1\in G(I)$ such that $w_1>v$ and
$w_1:v=x_{i_{h-1}}$. Hence $w_1=x_{i_1}^{\alpha_{i_1}}\cdots
x_{i_{h-2}}^{\alpha_{i_{h-2}}}
x_{i_{h-1}}x_{i_h}^{\alpha_{i_{h-1}}+\alpha_{i_h}-1}x_n^k$. Now,
since $u:w_1=x_{i_{h-1}}^{\alpha_{i_{h-1}}+\alpha_{i_h}-1}$, there
exists $w_2>w_1$ such that $w_2:w_1=x_{i_{h-1}}$. Hence
$w_2=x_{i_1}^{\alpha_{i_1}}\cdots x_{i_{h-2}}^{\alpha_{i_{h-2}}}
x_{i_{h-1}}^2x_{i_h}^{\alpha_{i_{h-1}}+\alpha_{i_h}-2}x_n^k$.
Continuing in the same way,  there exists $w_{\alpha_{i_{h-1}}}\in
G(I)$ such that $w_{\alpha_{i_{h-1}}}=x_{i_1}^{\alpha_{i_1}}\cdots
x_{i_{h-2}}^{\alpha_{i_{h-2}}}
x_{i_{h-1}}^{\alpha_{i_{h-1}}}x_{i_h}^{\alpha_{i_h}}x_n^k$, as
desired.

By what we have shown, it follows that
$I(P_{\{n\}})=(x_1,\ldots,x_{n-1})^{d-k}$. Hence $I(P_{\{n\}})$ has
a linear resolution. So by \cite[Proposition 2.4]{BaHe}, $I$ is
polymatroidal.
\end{proof}

\begin{prop}\label{dimthree} Let $I\subset S=K[x_1,x_2,x_3]$ be a monomial
ideal generated in a single degree. The following conditions are
equivalent:
\begin{enumerate}[\upshape (a)]
\item $I$ is polymatroidal.
\item $I$ has linear quotients  with respect to the reverse lexicographical ordering
  of the minimal generators induced by every ordering of variables.
\end{enumerate}
\end{prop}

\begin{proof}
The implication (a) $\implies$ (b) is known.

(b)$\implies$ (a): Let $\deg_{x_1}(u)>\deg_{x_1}(v)$ and
$\deg_{x_2}(u)<\deg_{x_2}(v)$.
 We have two cases:
\begin{enumerate}[\upshape (i)]
\item  If $u:v=x_1^rx_3^s$ for integers  $r,s>0$. Then $v:u=x_2^t$
for integer $t>0$. We consider the ordering $x_3>x_2>x_1$, so $v>u$.
Hence  by the assumption there exists $w\in G(I)$ such that $w>u$
and $w:u=x_2$. So $w=(u/x_1)x_2\in G(I)$.
\item If $u:v=x_1^r$ for integer $r>0$.
\begin{enumerate}[\upshape -]
\item Suppose $\deg_{x_3}(u)=\deg_{x_3}(v)$.  We assume that
$x_3>x_1>x_2$, so $u>v$. Hence  by the assumption there exists
$w_1\in G(I)$ such that $w_1>v$ and $w_1:v=x_1$. So
$w_1=(v/x_2)x_1$. Now since $u>w_1$ and $u:w_1=x_1^{r-1}$, it
follows that there exists $w_2>w_1$ such that $w_2:w_1=x_1$. Hence
$w_2=(w_1/x_2)x_1$. So $u>w_2$ and $u:w_2=x_1^{r-2}$. Continuing in
the same way, there exists $w_{r-1}=(w_{r-2}/x_2)x_1\in G(I)$, such
that $u> w_{r-1}$ and $u:w_{r-1}=x_1$. Thus $u=(w_{r-1}/x_2)x_1$.
Hence $(u/x_1)x_2=w_{r-1}\in G(I)$.

\item Suppose $\deg_{x_3}(u)<\deg_{x_3}(v)$. We assume that
$x_1>x_3>x_2$, so $u>v$. Hence by the assumption there exists
$w_1\in G(I)$ such that $w_1>v$ and $w_1:v=x_1$. So either
$w_1=(v/x_2)x_1$ or $w_1=(v/x_3)x_1$. Now since $u>w_1$ and
$u:w_1=x_1^{r-1}$, it follows that there exists $w_2>w_1$ such that
$w_2:w_1=x_1$. So either $w_2=(w_1/x_2)x_1$ or $w_2=(w_1/x_3)x_1$.
Hence $u>w_2$ and $u:w_2=x_1^{r-2}$. Continuing in the same way,
there exists $w_{r-1}\in G(I)$, such that $u>w_{r-1}$ and
$u:w_{r-1}=x_1$. Hence either $u=(w_{r-1}/x_2)x_1$ or
$u=(w_{r-1}/x_3)x_1$. Therefore either $(u/x_1)x_2\in G(I)$ or
$(u/x_1)x_3\in G(I)$.
\end{enumerate}
\end{enumerate}
\end{proof}

 Let $M_d$ denote the set of all monomials of degree $d$ in the
polynomial ring $S=k[x_1,\ldots,x_n]$.
 We order the monomials
lexicographically by the ordering $x_1 >_{lex} x_2 >_{lex}\cdots
>_{lex} x_n$.

\begin{prop}\label{lex}
Let $v=x_1^{b_1}\cdots x_n^{b_n}\in M_d$ and $I=(u\in M_d\;|\;
u\geq_{lex} v)$ be a monomial ideal in $S=K[x_1,\ldots,x_n]$. Then
the following conditions are equivalent:
\begin{enumerate}[\upshape (a)]
  \item $I$ is polymatroidal.
  \item $I$ has linear quotients with  respect to the reverse lexicographical ordering
  of the minimal generators induced by every ordering of variables.
\end{enumerate}
\end{prop}

\begin{proof}
The implication (a) $\implies$ (b) is known.

(b)$\implies$ (a): We consider $x_n>x_j>x_1$ for $j\neq 1,n$. Since
$\deg_{x_1}(v)<\deg_{x_1}(x_1^{b_1+1}x_n^{d-(b_1+1)})$, we have
$v>x_1^{b_1+1}x_n^{d-(b_1+1)}$. Assume that $\n=(x_2,\ldots,x_n)$.
We have the following cases:

\textbf{Case 1.} $v:x_1^{b_1+1}x_n^{d-(b_1+1)}=x_n$. Then
$v=x_1^{b_1}x_n^{d-b_1}$. Hence
 $I=\sum_{i=b_1}^{d}x_1^{i}\n^{d-i}$, so $I$ is polymatroidal.

\textbf{Case 2.} $v:x_1^{b_1+1}x_n^{d-(b_1+1)}=x_{j}$, where $j\neq
1,n$. Hence $v=x_1^{b_1}x_{j}x_n^{d-(b_1+1)}$. Therefore
$I=\sum_{i=b_1+1}^{d}x_1^{i}\n^{d-i} +J$, where

$$J:=(w\in M_d\;|\; deg_{x_1}(w)=b_1 \;\text{and}\; w\geq_{lex} v).$$
Now we want to prove that $I$ is polymatroidal.
\begin{enumerate}[\upshape (i)]
\item
Let $w_1\in x_1^t\n^{d-t}$, where $b_1+1\leq t\leq d$ and $w_2\in
J$. Assume that $\deg_{x_i}(w_1)>\deg_{x_i}(w_2)$. If $i\neq 1$,
then  $(w_1/x_i)x_k\in I$, for all $k$. If $i=1$ and there exists
$1<h\leq j$ such that $x_h|w_1$. Then  $(w_1/x_i)x_k\in I$, for all
$k$. Otherwise,  $i=1$ and $x_h\nmid w_1$, for all  $1<h\leq j$. Now
since $w_2\geq_{lex} v$, it follows that there exist $1< l\leq j$
such that $x_l|w_2$. Hence $(w_1/x_i)x_l\in I$.
\item
Let $w_1\in J$, $w_2\in x_1^{t}\n^{d-t}$, where  $b_1+1\leq t\leq
d$. Assume $\deg_{x_i}(w_1)>\deg_{x_i}(w_2)$, so $(w_1/x_i)x_1\in
I$.
\item
Let $w_1\in J$, $w_2\in J$ and $\deg_{x_i}(w_1)>\deg_{x_i}(w_2)$.
\begin{enumerate}[\upshape -]
\item If $\deg_{x_j}(w_2)>\deg_{x_j}(w_1)$, then $(w_1/x_i)x_j\in I$.

\item If $\deg_{x_j}(w_2)<\deg_{x_j}(w_1)$ and $i\neq j$. Since
$x_j|(w_1/x_i)x_k$ for all $k$, then $(w_1/x_i)x_k\in I$ for all
$k$.

\item If  $\deg_{x_j}(w_2)<\deg_{x_j}(w_1)$, $i=j$ and $x_j|w_2$. Since
 $\deg_{x_j}(w_1)\geq 2$, then $(w_1/x_i)x_k\in I$ for all $k$.

\item If  $\deg_{x_j}(w_2)<\deg_{x_j}(w_1)$, $i=j$, $x_j\nmid w_2$ and there exists $1<l<j$ such that
$\deg_{x_l}(w_2)>\deg_{x_l}(w_1)$. Then $(w_1/x_i)x_l\in I$.

\item If  $\deg_{x_j}(w_2)<\deg_{x_j}(w_1)$, $i=j$, $x_j\nmid w_2$ and  for all $1<l<j$,
 $\deg_{x_l}(w_2)\leq\deg_{x_l}(w_1)$, then $(w_1/x_i)x_k\in I$ for all
$k$. Note that Since $x_j\nmid w_2$, then there exists $1<l'<j$ such
that $x_{l'}|w_2$ and hence $x_{l'}|w_1$.
\end{enumerate}
\end{enumerate}
\textbf{Case 3.} $v:x_1^{b_1+1}x_n^{d-(b_1+1)}=x_{j_1}\cdots
x_{j_l}$, where $l>1$ and $j_r\neq n$ for each $r=1,\ldots,l$. Then
there exists $w\in G(I)$ such that $w>x_1^{b_1+1}x_n^{d-(b_1+1)}$
and $w:x_1^{b_1+1}x_n^{d-(b_1+1)}=x_{j_s}$ for some $1\leq s\leq l$.
Now, since we consider the ordering $x_n>x_j>x_1$ for $j\neq 1,n$,
it follows that $w= x_1^{b_1}x_{j_s}x_n^{d-(b_1+1)}$. This is a
contradiction, since $w>_{lex} v$.
\end{proof}

\begin{defn}
A \textit{lexsegment} (of degree $d$) is a subset of $M_d$ of the
form
$$\mathcal{L}(u, v)=\{w\in M_d \;|\;u\geq_{lex}w\geq_{lex}v\}$$
 for some
$u,v\in M_d$ with $u\geq_{lex} v$. A lexsegment $L$ is called
\textit{completely lexsegment} if all the iterated shadows of L are
again lexsegments. An ideal spanned by a completely lexsegment is
called a completely lexsegment ideal.
\end{defn}
We recall that the shadow of a set $T$ of monomials is the set
$\Shad(T)=\{vx_i\;|\;v\in T,1\leq i\leq n\}$. The $i$-th shadow is
recursively defined as $\Shad^i(T)=\Shad(\Shad^{i-1}(T))$.

\begin{prop}\label{comlex}
Let $I\subset K[x_1,\ldots,x_n]$ be a monomial ideal generated in
degree $d$ and suppose that $I$ is a completely lexsegment ideal.
Then the following conditions are equivalent:
\begin{enumerate}[\upshape (a)]
  \item $I$ is polymatroidal.
  \item $I$ has linear quotients with  respect to the reverse lexicographical ordering
  of the minimal generators induced by every ordering of variables.
\end{enumerate}
\end{prop}

\begin{proof}
The implication (a) $\implies$ (b) is known.

(b)$\implies$ (a): Let $I=(\mathcal{L}(u,v))$, where
$u=x_1^{a_1}\cdots x_n^{a_n}$ and $v=x_1^{b_1}\cdots x_n^{b_n}$. By
\cite[Theorem 1.3]{ArNeHe}, since $I$ has a linear resolution, one
of the following conditions holds:
\begin{enumerate}[\upshape (i)]
\item  $u=x_1^ax_2^{d-a}$ and $v=x_1^ax_n^{d-a}$ for some $a$, $0<a\leq d$.
\item  $b_1\leq a_1-1$.
\end{enumerate}
Suppose (i) holds. Since $\deg_{x_1}(u)=\deg_{x_1}(v)$, so
$I=x_1^{a_1}(x_2,\ldots,x_n)^{d-a_1}$. Obviously $I$ is
polymatroidal.

Now let (ii) holds. Assume that $a_3+a_4+\cdots+a_n\neq 0$. We
consider the ordering $x_1>x_2>x_3>\cdots >x_n$. Suppose
$w=x_1^{a_1-1}x_2^{d-a_1+1}$. Since $w\in I$, $w>u$,
$w:u=x_2^{d-a_1-a_2+1}$ and also $d-a_1-a_2+1>1$, it follows that
there exist $w'\in I$ such that $w'>u$ and $w':u=x_2$. Hence
$w'=(u/x_i)x_2$ for some $3\leq i\leq n$. So $w'>_{lex} u$, which is
a contradiction.

Now let $a_i=0$ for $i=3,\ldots,n$. Hence $u=x_1^{a_1}x_2^{d-a_1}$
and so  $x_1^{b_1+1}x_n^{d-(b_1+1)}\in I$. We consider $x_n>x_j>x_1$
for $j\neq 1,n$, hence $v>x_1^{b_1+1}x_n^{d-(b_1+1)}$. Assume that
$\n=(x_2,\ldots,x_n)$. We have the following cases:

\textbf{Case 1.} $v:x_1^{b_1+1}x_n^{d-(b_1+1)}=x_n$. Then
$v=x_1^{b_1}x_n^{d-b_1}$. Hence
 $I=\sum_{i=b_1}^{a_1}x_1^{i}\n^{d-i}$, so $I$ is polymatroidal.

\textbf{Case 2.} $v:x_1^{b_1+1}x_n^{d-(b_1+1)}=x_{j}$, where $j\neq
1,n$. Hence $v=x_1^{b_1}x_{j}x_n^{d-(b_1+1)}$. Therefore
$I=\sum_{i=b_1+1}^{a_1}x_1^{i}\n^{d-i} +J$, where

$$J:=(w\in M_d\;|\; deg_{x_1}(w)=b_1 \;\text{and}\; w\geq_{lex} v).$$
With the same arguments as used in the proof of  case 2 of
Proposition \ref{lex},  $I$ is polymatroidal.

\textbf{Case 3.} $v:x_1^{b_1+1}x_n^{d-(b_1+1)}=x_{j_1}\cdots
x_{j_l}$, where $l>1$ and $j_r\neq n$ for each $r=1,\ldots,l$. With
the same arguments as used in the proof of case 3 of Proposition
\ref{lex}, $I$ is polymatroidal.

\end{proof}

Based on Remark \ref{two variable}, Proposition \ref{quadratic},
Proposition \ref{pure power}, Proposition \ref{dimthree},
Proposition \ref{lex}, Proposition \ref{comlex} and based on
experimental evidence we are inclined to make the following:

\begin{conj}\label{conj}
Let $I\subset S$ be a monomial ideal generated in a single degree.
Then the following conditions are equivalent:
\begin{enumerate}[\upshape (a)]
  \item $I$ is polymatroidal.
  \item $I$ has linear quotients with respect to the reverse lexicographical ordering of the minimal
  generators induced by every ordering of variables.
\end{enumerate}
\end{conj}

Let $I\subset S$ be a monomial ideal minimally generated by
$u_1,\ldots,u_r$. We say that $I$ has \emph{quotients with linear
resolution} with respect to the ordering $u_1,\ldots,u_r$ whenever
$I$ has a linear resolution and, moreover, for all $j=2,\ldots,r$,
the colon ideal $(u_1,\ldots,u_{j-1}):u_j$ has a linear resolution.

\begin{rem}
It is clear that if a monomial ideal  generated in a single degree
has linear quotients with respect to the ordering $u_1,\ldots,u_r$
of minimal generators then it also has quotients with linear
resolution with respect to the ordering $u_1,\ldots,u_r$. So
polymatroidal ideals have  quotients with linear resolution with
respect to the (reverse) lexicographical ordering of the minimal
generators induced by every ordering of variables.

It is natural to ask whether  Conjecture \ref{conj} can be weakened
to quotients with linear resolution property with respect to the
(reverse) lexicographical
 ordering. There is a negative
answer to this question in a general. For example the ideal
$I=(x_1x_3^2,x_1^2x_3,x_1x_2x_3, x_2^2x_3)$ has quotients with
linear resolution with respect to the lexicographical
 and  reverse lexicographical
 ordering of the minimal generators induced by every ordering of
 variables, but it is not polymatroidal. We note that $I$ dose not have
 linear quotients with respect to the
lexicographical  ordering of the minimal generators induced by the
ordering $x_3>_{lex}x_2>_{lex}x_1$ and   the reverse lexicographical
ordering of the minimal generators induced by the ordering
$x_3>x_2>x_1$ of variables.
\end{rem}

\section*{\bf Acknowledgments}
The authors would like to thank referees for valuable and useful
comments regarding this paper.

\end{document}